\documentclass[11pt]{article}

\usepackage{lscape,tikz}
\usepackage{epsfig,cancel,ulem, amsthm,amssymb}
\definecolor{darkblue}{rgb}{0.2,0.2,0.71}
\usepackage{todonotes}
\usepackage{graphicx}%
\usepackage{lscape}

\usepackage{mathtools}
\usepackage{multirow}
\usepackage{multicol}
\usepackage{tabularx,ragged2e,booktabs,caption,amssymb}
\usepackage{tikz}

\usepackage{rotating}

\usepackage{lipsum}

\usetikzlibrary{decorations.markings}
\tikzstyle{vertex}=[circle, draw, inner sep=0pt, minimum size=3pt]
\newcommand{\vertex}{\node[vertex]}
\newtheorem{thm}{Theorem}[section]
 \newtheorem{cor}[thm]{Corollary}
 \newtheorem{lem}[thm]{Lemma}
 \newtheorem{prop}[thm]{Proposition}

 \newtheorem{example}[thm]{Example}
 
\newtheorem{proposition}[thm]{Proposition}

\theoremstyle{theorem}
\newtheorem{theorem}{Theorem}

\theoremstyle{definition}
\newtheorem*{definition}{Definition}
\newtheorem*{remark}{Remark}

  \newcommand{\s}{{S}}
 \newcommand{\A}{\mathcal{A}}

  \newcommand{\h}{\mathfrak{h}}

 \newcommand{\g}{\mathfrak{g}}

 \newcommand{\ad}{\textrm{ad}}

\newcommand{\bil}[2]{{\langle #1 | #2\rangle}}

 \newcommand{\ew}{\mathcal{W}}
 \newcommand{\eh}{\mathcal{H}}
\def\bt{\begin{theorem}}
\def\et{\end{theorem}}
\def\bc{\begin{corollary}}
\def\ec{\end{corollary}}
\def\bx{\begin{example}}
\def\ex{\end{example}}
\def\bxr{\begin{exercise}\small}
\def\exr{\end{exercise}}
\def\bl{\begin{lemma}}
\def\el{\end{lemma}}
\def\bd{\begin{definition}}
\def\ed{\end{definition}}
\def\bp{\begin{proposition}}
\def\ep{\end{proposition}}

\def\br{\begin{remark}}
\def\er{\end{remark}}

\def\be{\begin{equation}}
\def\ee{\end{equation}}

\def\&{\hspace{-15pt}&}
\def\bea{\begin{eqnarray}}
\def\eea{\end{eqnarray}}

\begin{document}

\title{ Weights of Semiregular   Nilpotents  in Simple Lie Algebras of  D Type}
\markright{Submission}
\author{Yassir Dinar}

\maketitle
\begin{abstract}
We compute the  weights of the adjoint action of  semiregular $sl_2$-triples in simple Lie algebras of type $D_n$ using mathematical induction.
\end{abstract}




\section{Introduction}
Let  $\g$ be a complex simple Lie algebra of rank $r$ with Lie bracket denoted  $[.,.]$. Consider the adjoint map $\ad:\g\to End(\g)$ defined by $\ad_g(x):=[g,x]$ for all $g,x\in \g$. Then  $g$ is called nilpotent if $\ad_g$ is nilpotent in $End(\g)$. We fix a nilpotent  element  $e$  in $\g$. Then,  using Jacobson-Morozov theorem, we also  fix a semisimple element  $h$ and a nilpotent element $f$ in $\g$ such that the set $\A=\{e,h,f\}$ forms a $sl_2$-triple with relations
\begin{equation}\label{sl2:relation} [h,e]=2 e,\quad [h,f]=-2f,\quad [e,f]= h.
\end{equation}

It follows from representation theory of $sl_2$ algebra  that the restriction of the adjoint map to  the subalgebra generated by $\A$ leads  to the decomposition of $\g$ into irreducible $\A$-submodules
\begin{equation}\label{decompo}
\g=\bigoplus_{i=1}^{n} V^i.
\end{equation}
Here   $n$ equals $\dim (\ker \ad_e)\geq r$ and  $\dim V^i=2\eta_i+1$ where $\eta_i$ is an integer or a half integer. Then the eigenvalues with multiplicities of the action of $\ad_h$ on $\g$ is the multiset formed by  the union of the arithmetic series $2 \eta_i-2j$ where $j=0,1,\cdots, 2\eta_i$, $i=1,\ldots,n$. We define\textbf{ the weight partition $\ew$} associated to $e$ as the {\bf multiset}
\begin{equation}
\ew:=[\eta_1,\eta_2,\ldots,\eta_r].
\end{equation}
We  assume throughout the article that the numbers in $\ew$ are given in non increasing order, i.e. $\eta_i\geq \eta_j$ when $i<j$. It is known that taking any other nilpotent element in the  orbit $\mathcal{O}_e$ of $e$ under the adjoint group action gives equal weight partition.  In this article we calculate  $\ew$ for certain infinite families  of nilpotent orbits.

A nilpotent element $g\in \g$  is called regular if and only if $\dim (\ker ad_e)=r$. Any simple Lie algebra possesses regular nilpotent elements and all of them lie in one orbit. Let us assume  that $e$ is  a regular nilpotent element.  In this case, Kostant proved that the numbers given in $\ew$ is just  what is known as the exponents of the Lie algebra $\g$ \cite{kostBetti}. Note that exponents of $\g$ was calculated  by using various involving  tricks  from   invariant theory to   analysing  properties of  Coxeter elements in the underline Weyl group \cite{JHum}.   The connection to invariant theory is govern  by  Chevalley's theorem which states that    the invariant ring  of the adjoint group action on $\g$ is a polynomial ring with $r$ homogeneous  generators of degrees $\eta_i+1$.
While, it is known that the eigenvalues of  a Coxeter element  acting   on a Cartan subalgebra are  $\omega^{\eta_i}$ where  $\omega$ is $(\kappa+1)$th root of unity, $\kappa=\max \ew$.  However, there is rather an  elementary   procedure  to calculate the weight partition of $e$  which depends on  defining and partitioning the underline set of  positive roots using   $h$ (\cite{COLMC},section 4.4). In  next section, we will explain how we can use the same procedure to obtain the weight partitions for  a class of nilpotent orbits  known as  distinguished nilpotent orbits. This procedure will be used in this article to get our  main results.

In the case $e$ is a subregular nilpotent element, i.e. $\dim (\ker ad_e)=r+2$. Also, subregular nilpotent elements exist in all simple Lie algebras and belong to the same orbit. The partition $\ew$ was obtained by Slodowy (\cite{sldwy1}, section 7.4). The structure of the numbers in $\ew$ was essential for him to  prove that the restriction of the adjoint quotient map to $e+ \ker \ad \,f$ is a semiuniversal deformation of simple hypersurface singularity of the same type (\cite{sldwy1}, section 8.3).

For arbitrary nilpotent element $e$, $\ew$ appears in the theory of $W$-algebras. It is known that preforming Drinfed-Sokolov reduction associated to the nilpotent element $e$, we obtain a classical $W$-algebra which consists of Virasoro density and $n-1$ conformal primary fields. The weights of these primary fields are $\eta_i+1$, $i=2,\ldots,n$ \cite{feher}.
Our interest in $\ew$ was initiated by the finding that   the degrees of a  Frobenius manifold that can be obtained from  Drinfeld-Sokolov reduction associated to $e$ can be read from the set $\ew$ \cite{mypaper1}, \cite{mypaper5}.

From what is mentioned above we see that weight partitions   are essential for understanding and analysing the  geometric and algebraic structure obtained using nilpotent elements. However, even nilpotent elements in simple Lie algebra have been classifies through their orbits, there is  a gap in the literature for the value of $\ew$ in case $e$ not  regular or subregular.  In this article,  we calculate the weight  partition  of  all semiregular   nilpotent elements  in Lie algebra of type $D_r$. The methods used in this article  are elementary, using only mathematical induction and general theory of Lie algebra. However, these methods  can be used for  other distinguished nilpotent elements.

For the case  $\g$ is a Lie algebra of type $D_r$, semiregular nilpotent orbits are  denoted  $D_r(a_k)$ where   $k$ is an integer and  $0\leq k<\frac{r}{ 2}$ (more details are given below). Note that  a regular nilpotent element is of type $D_r(a_0)$ and subregular one is of type $D_r(a_1)$.    In this article we will prove the following
\begin{thm}
The weight partition  $\ew_{k,r}$ for a nilpotent element of type $D_r(a_k)$, $0\leq k<\lfloor \frac{r}{2}\rfloor$, is the multiset formed as the sum   of the multisets  $ [1,3,5,\cdots,2r-2k-3]$, $[1,3,5,\cdots,2k-1] $ and  $[r-1,r-2,r-3,r-4,\cdots,r-2k,r-2k-1] $.
\end{thm}

We illustrate by example  the sum and difference  of two multisets as this operations are needed in the article. Simply,  the sum of the two multisets $[4,2,2,2,1,1]$ and $[5,4,2,2]$ is $[5,4,4,2,2,2,2,2,1,1]$ and their difference $[4,2,2,2,1,1]- [5,4,2,2]$ equals $[2,1,1]$.

All information needed to verify the   findings of this study are included within the article.

\section{The height  partition of  distinguished nilpotents}

In this section we recall the classification of nilpotent elements and the definition of distinguished nilpotent elements. Then  we will use the same logic given in (\cite{COLMC}, section 4.4) for the case of regular nilpotent elements  to derive  a procedure   to find  weight partitions  for arbitrary   distinguished nilpotent elements. We keep the notations given in the introduction. All facts about nilpotent elements given in this section is obtained from the excellent book  \cite{COLMC}.

Let us consider the  Dynkin grading associated to $e$
 \begin{equation}
\g=\oplus_{i\in \mathrm{Z}} \g_i;~~~\g_i:=\{g\in\g: [h, g]=i g\};~~[\g_i,\g_j]\subseteq \g_{i+j}.
\end{equation}
Then a nilpotent element $e$ is called distinguished if and  only if $\dim \g_0=\dim \g_2$.

Let us fix a Cartan subalgebra $\h\subset \g_0$  containing $h$ and denote $\Phi$ the associated root system. We  define a set of positive roots by $\Phi^+=\{\alpha\in \Phi:\alpha(h)\geq 0\}$.  Let $\Delta=\{\alpha_1,\alpha_2,\cdots, \alpha_r\}$ be the set of simple roots in $\Phi^+$. The weighted Dynkin diagram of $e$ is the Dynkin diagram of $\g$ where we assign the value $\alpha_i(h)$ to the node of $\alpha_i$. By a nilpotent orbit $\mathcal{O}_e$ of $e$, we mean the orbit of $e$ under the adjoint group action. It turns out that the weighted Dynkin diagram of $e$ completely characterizes the nilpotent orbit  $\mathcal O_e$, i.e. $e'\in \mathcal{O}_e$ if and only if $e'$ has the same weighted Dynkin diagram  as  $e$ (see \cite{COLMC} for details). This leads to the classification of nilpotent elements by their orbits. In this classification, distinguished nilpotent orbits are always denoted $Z_r(a_i)$ where $Z$ is the type of  $\g$ and $i$ is the number of vertices of weight 0 in the corresponding weighted  Dynkin diagram.  If there is another orbit of the same number $i$ of 0's then the notation $Z_r(b_i)$ is used. If $e$ is distinguished then $\alpha_i(h)\in \{0,2\}$ for every $i$. Thus, the numbers in $\ew$ are all integers and  $\g_i=0$ if $i$ is odd \cite{COLMC}.

{\bf We assume  for the rest of this article that $e$ is a distinguished nilpotent element}. We set
\begin{equation}
    R:=\dim \g^+,~~ \g^+:=\{v\in \g: ad_h(v)=\lambda v \text{ and } \lambda>0\}.
\end{equation}
Then using representation theory of $sl_2$ algebras, we can consider $\ew$  as a  partition of the natural number $R$, i.e. $R=\sum_{i=1}^{n}\eta_i$. On the other hand, if we define the \textbf{height partition} $\eh$ associated of  $e$ as the multiset  \begin{equation}
 \eh:=[\dim{\g_2},\dim{\g_4},\ldots].
  \end{equation}
Then it turns out that  $\eh$ is another partition of $R$, i.e. $\sum_{i>0} \dim \g_i=\sum_{i=1}^{n}\eta_i$. Note that the numbers for both $\eh$ and $\ew$ are given in a non increasing order. Then, using again representation theory of $sl_2$ algebra, it is easy to observe that $\ew$ is the {\bf transpose (or conjugate) partition} of $\eh$. This means that given the partition $\eh$, then $\ew$ is found by the formula
\begin{equation}\label{dualprn}
\eta_j=|\{i|\dim \g_i\geq j\}|.
\end{equation}

We illustrate another way to define  $\eh$ which depends on  studying the action of $h$ on $\Phi^+$. Consider  Cartan decomposition
\begin{equation}\label{cartan:decom}
\g=\h\oplus\bigoplus_{\beta\in \Phi}\g_\beta
\end{equation}
and define  the set $\rho_i$ of roots of height $i$ associated to $e$ by
\begin{equation}\label{hights}
\rho_i=\{\beta\in \Phi:\beta(h)=2i\}=\{  \sum m_j \alpha_j\in \Phi: \sum_{\alpha_j(h)=2} m_j=i\}.
\end{equation}
Then  $\g_i=\oplus_{\{\alpha\in \rho_i\}}\g_\alpha$ and   $\dim\g_{i}=|\rho_i|$. Hence,  $\eh$ can be redefined as
\begin{equation}
\eh=[ |\rho_1|,|\rho_2|,\ldots].
\end{equation}
  This also illustrates that  finding the height and  weight partitions depends only on the semisimple element $h$   which is uniquely specified by  the weighted Dynkin diagram.

\section{Lie algebra of type $D_r$ and semiregular nilpotents}

 {\bf We  assume for  the reminder of this article that $\g$ is a Lie algebra of type $D_r$}.   We denote the simple roots of $D_r$ by  $\alpha_1,\alpha_2,...,\alpha_r$. We  order them as illustrated on  the  Dynkin diagram of $\g$ given in figure \ref{fig:WDD}. Our order for the simple roots are not the standard in the literature (\cite{JHum},\cite{car},\cite{bour}) but we found it very convenient for using mathematical induction.

We denote  the set of positive  roots of $D_r$ by $\Phi_{r}^+ $. Then, it consists of the following  $r^2-r$ roots \cite{bour}: In addition to $ \alpha_1$, we have
\begin{eqnarray*}
\alpha_{i+1}+\cdots+\alpha_{j}& ;& i<j\\
\alpha_1+\alpha_2+(2\alpha_3+.....+2\alpha_{i})+\alpha_{i+1}+....+\alpha_j&;&~3\leq i< j \leq r\\\
\alpha_1+\alpha_2+\cdots+\alpha_{j}&;&~ 3\leq j\\
 \alpha_1+\alpha_3+\cdots+\alpha_{j}&;&~ 3 \leq j
 \end{eqnarray*}
The root $\alpha_1+\alpha_2+(2\alpha_3+.....+2\alpha_{r-1})+\alpha_{r}$ is called the highest root of $\Phi_r^+$ and will be denoted $\gamma_r$.

 Recall that a subalgebra $\s$ of $\g$ is called regular if it has the form  $\s=\h\oplus\bigoplus_{\beta\in \Theta}\g_\beta$ for some  $\Theta \subseteq \Phi_r$. Then a nilpotent element is  called semiregular if its orbit has no intersection with any proper regular subalgebra of $\g$.
  The semiregular nilpotent orbits in Lie algebra of type $D_r$ are the distinguished nilpotent orbits  $D_r(a_k)$ where $0\leq k< \lfloor \frac{r}{2}\rfloor$ \cite{arbour}, here $\lfloor\cdot \rfloor$ is the floor function. When  $\g$ is the special orthogonal Lie algebra $so_{2r}$, $D_r(a_k)$ corresponds to the partition $[2r-2k-1,2k+1]$ of $2r$ \cite{arbour}. The weighted Dynkin diagram of $D_r(a_k)$ has all weights equal 2 except those   $\alpha_s$ for  $s=2m+1, ~0<m\leq k$ (see figure \ref{fig:WDD}) . The orbit $D(a_0)$ is the regular nilpotent orbit while $D_r(a_1)$ is the subregular nilpotent orbit.

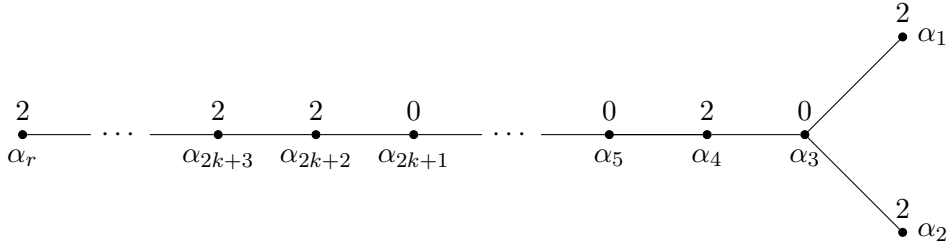
\begin{figure}
\centering \begin{tikzpicture}[x=1.3cm, y=1.3cm]
	\vertex[fill] (a1) at (2,1) [label=right:$\alpha_1$,label=above:2] {};
	\vertex[fill] (a2) at (2,-1) [label=right:$\alpha_2$,label=above:2] {};
	\vertex[fill] (a3) at (1,0) [label=below:$\alpha_3$,label=above:0] {};
	\vertex[fill] (a4) at (0,0) [label=below:$\alpha_4$,label=above:2] {};
	\vertex[fill] (a5) at (-1,0) [label=below:$\alpha_5$,label=above:0] {};
  \node        (q_dots) at (-2,0) {$\cdots$};

\vertex[fill] (ak) at (-3,0) [label=below:$\alpha_{2k+1}$,label=above:0] {};
\vertex[fill] (ak1) at (-4,0) [label=below:$\alpha_{2k+2}$,label=above:2] {};
\vertex[fill] (ak2) at (-5,0) [label=below:$\alpha_{2k+3}$,label=above:2] {};
 \node        (q_dots1) at (-6,0) {$\cdots$};
 \vertex[fill] (an) at (-7,0) [label=below:$\alpha_{r}$,label=above:2] {};
	\path
		
(a1) edge (a3)
(a2) edge  (a3)
(a3) edge  (a4)
(a4) edge  (a5)
(a4) edge (q_dots)
(q_dots) edge (ak)
(ak) edge (ak1)
(ak1) edge (ak2)
(ak2) edge (q_dots1)
(q_dots1) edge (an)


	;
\end{tikzpicture}

\caption{Weighted Dynkin Diagram of $D_r(a_k)$} \label{fig:WDD}
\end{figure}

   In this paper we will always {\bf consider  $\Phi_{r}^+$  as a subset of the positive roots $\Phi_{r+1}^+$ of $D_{r+1}$}. For $\alpha,\beta \in \Phi_r^+$, {\bf we say that $\alpha$ divides $\beta$, in notations $\alpha|\beta$}, if there exist $\gamma \in \Phi_{r}^+ \cup \{0\}$ such that $\gamma+\alpha=\beta$. Then, we have the following

 \begin{lem}\label{basic1}
 The set $\Phi_{r+1}^+$ is  the  union of the sets $\Phi_{r}^+$, $\{\alpha_{r+1}\}$, $\{\beta+\alpha_{r+1}: \beta\in \Phi_{r}^+ \, \wedge\, \alpha_{r}|\beta\}$ and $\{\gamma_r+\alpha_{r+1},\gamma_r+\alpha_{r}+\alpha_{r+1}\}$, where $\gamma_r$ is the highest root in $\Phi_{r}^+$.
 \end{lem}
 \begin{proof}
 Recall that  the Dynkin diagram encodes the values of a bilinear form $\bil . .$ on the set of roots where  $\bil {\alpha_i}{\alpha_i}=2$ and $\bil {\alpha_i}{\alpha_j}=-1$ if and only if there is a link between $\alpha_i$ and $\alpha_j$.  Moreover, for a root $\beta$, we have  $\beta+\alpha_j$ is a root if and only if $\bil \beta {\alpha_j}<0$. We observe that adding  $\alpha_{r+1}$ to the Dynkin diagram of $D_r$ then $\bil {\alpha_{r+1}}{\alpha_i}$ is nonzero only when $i=r$ and in this case the value is $-1$. Hence, for $\beta\in \Phi_r^+$,  if  $\alpha_r|\beta $ then $\bil \beta {\alpha_{r+1}}=-1$ and $\beta+\alpha_{r+1}$ is a root.  Note that $\alpha_r\nmid \gamma_r$ but $\gamma_r$ leads to two roots  $\gamma_r+\alpha_{r+1}$ and $\gamma_r+\alpha_r+\alpha_{r+1}$.
 \end{proof}

 Following  the definitions given  in the last section (see equation  \eqref{hights}), for a nilpotent element of type $D_r(a_k)$, let  $\rho_{i,k,r}$ denotes  the set of roots of height $i$, $\eh_{k,r}$ denotes  the height partition and $\ew_{k,r}$ denotes  the weight partition. Thus
\begin{equation}\label{hights:def}
\rho_{i,k,r}=\{\sum m_j \alpha_j\in \Phi_r^+:\sum_{j\neq 2t+1;0<t\leq k}^r m_j=i\}, ~\eh_{k,r}=[ |\rho_{1,k,r}|,|\rho_{2,k,r}|,\ldots].
 \end{equation}

   From lemma \ref{basic1} we get the following very useful result.

 \begin{lem}\label{basic} Consider the set $\Phi_{r}^+$ as a subset of $\Phi_{r+1}^+$. Let  $m_r=\max \{i: \rho_{i,k,r}\neq \varnothing\} $  and define the set
 $\xi_{i,k,r}:=\{ \beta\in\rho_{i-1,k,r}:\alpha_r \text{ divides } \beta\}
$.
 Then
 \begin{enumerate}
     \item $\rho_{1,k,r+1}=\rho_{1,k,r}\cup  \{\alpha_{r+1}\}$.
     \item $\rho_{m_r+1,k,r+1}=\{\gamma_r+\alpha_{r+1}\}$ and $\rho_{m_r+2,k,r+1}=\{\gamma_r+\alpha_{r}+\alpha_{r+1}\}$.
     \item $\rho_{i,k,r}\subset \rho_{i,k,r+1}$ and
     $|\rho_{i,k,r+1}|= |\rho_{i,k,r}|+ |\xi_{i,k,r}|, ~1<i\leq m_r$.
 \end{enumerate}
 \end{lem}
\begin{proof}
 Note that the weighted Dynkin diagram of $D_{r+1}(a_k)$  is obtained from that of $D_r(a_k)$ by adding the root $\alpha_{r+1}$ with weight 2. Thus from lemma \ref{basic1}, the first consequent   is clear and  the second  follows from the fact that  the highest root $\gamma_r$ of $\Phi_r^+$ belongs to $\rho_{m_r,k,r}$ and also $\alpha_r$ has weight 2.  Obviously,  $\rho_{i,k,r+1}$ is a disjoint union of $A=\{\beta\in \rho_{i,k,r+1}: \alpha_{r+1} \nmid \beta\}$ and $B=\{\beta\in \rho_{i,k,r+1}: \alpha_{r+1} | \beta\}$. But then $A=\rho_{i,k,r}$ and $B=\{\beta+\alpha_{r+1}: \beta \in \rho_{i,k,r} \text{ and } \alpha_r| \beta\}=\xi_{i,k,r}$. This proves the  third consequent.
\end{proof}


 \section{Regular nilpotent element in $D_r$}

 We assume in this section that the nilpotent element $e$ is regular (of type $D_r(a_0)$). Then, as we mentioned in the introduction, the  weight partition $\ew_{0,r}$ is already known. However, we give a complete proof of how to find $\ew_{0,r}$ using the procedure  outlined in section 2.

 Observe that, in this case, $\alpha_i(h)=2$ for every simple root $\alpha_i$, and so $\rho_{i,0,r}=\{\sum m_j\alpha_j~:\sum m_j=i\}$

 \begin{prop}\label{reqnilpel} The set $\rho_{2r-3,0,r}$ contains only the highest root. All other sets  $\rho_{i,0,r},\, i<2r-3$ have exactly one root  divisible by $\alpha_r$, except  $\rho_{r-1,0,r}$ has  two roots divisible by $\alpha_r$. The partition $\eh_{0,r}$  is a difference of two multisets
 \[ \eh_{0,r}=[r,r,r-1,r-1,...,1,1] - [r,\lfloor {r\over 2} \rfloor,\lceil {r\over 2}\rceil].\]
here $\lfloor\cdot \rfloor$ and $\lceil \cdot \rceil$ are the floor and ceiling  functions, respectively.
 \end{prop}
 \begin{proof}
 The proof is by induction on $r$. The proposition is true when $r=4$ since
  \begin{eqnarray*}
  \rho_{1,0,4}&=&\{\alpha_1,\alpha_2,\alpha_3,\alpha_4\}\\
   \rho_{2,0,4}&=&\{\alpha_1+\alpha_3,\alpha_2+\alpha_3,\alpha_3+\alpha_4\}\\
    \rho_{3,0,4}&=&\{\alpha_1+\alpha_2+\alpha_3,\alpha_2+\alpha_3+\alpha_4,\alpha_1+\alpha_3+\alpha_4\}\\
     \rho_{4,0,4}&=&\{\alpha_1+\alpha_2+\alpha_3+\alpha_4\}\\
      \rho_{5,0,4}&=&\{\alpha_1+\alpha_2+2\alpha_3+\alpha_4\}
      \end{eqnarray*}
 and  $\eh_{0,4}=[4,3,3,1,1]$. We assume the statement is  true for $r$. Then from lemma \ref{basic} we have  $|\rho_{i,0,r+1}|=|\rho_{i,0,r}|+1$ when $1\leq i\leq 2r-3$ and $i\neq r$, while  $|\rho_{r,0,r+1}|=|\rho_{r,0,r}|+2$. Hence $\rho_{r,0,r+1}$ contains two roots divisible by $\alpha_{r+1}$. The highest root in $\rho_{2r-3,0,r}$  yields one root  in $\rho_{2r-2,0,r+1}$ and the  highest  root  of $\Phi_{r+1}^+$ in $\rho_{2(r+1)-3,0,r+1}$. This proves the first part of the proposition. Hence, when $r$ is even,  $\eh_{0,r+1}$ is obtained by adding 1 and 2 to the numbers in  $\eh_{0,r}$ subject to
 \begin{equation}\label{niseven}
 [\overbracket{\underbrace{r,r-1,r-1, \cdots,{r\over 2}+1,{r\over 2}+1}_\text{+1}}^{|\rho_{1,0,r}|,|\rho_{2,0,r}|,\cdots,|\rho_{r-1,0,r}|},\overbracket{\underbrace{{r\over 2}-1}_\text{+2}}^{|\rho_{r,0,r}|}
    ,{\underbrace{\overbracket{{r\over 2}-1,\cdots,2,2,1,1}^{|\rho_{r+1,0,r}|,\ldots,|\rho_{2r-3,0,r}|},0,0}_\text{+1}}]
\end{equation}
  One can check that after addition and  substituting $m=r+1$, we get the proposed partition \eqref{nisodd} for $\eh_{0,m}$ in case $m$ is odd. Similarly, $\eh_{0,m+1}$ is obtained from $\eh_{0,m}$ by the following addition
\begin{equation}\label{nisodd}
\begin{split}
 [\overbracket{\underbrace{m,m-1,m-1,\cdots,{m+3\over 2},{m+3\over 2},{m+1\over 2}}_\text{+1}}^{|\rho_{1,0,m}|,|\rho_{2,0,m }|,\ldots,|\rho_{m-1,0,m}|}  ,\underbrace{{m-1\over 2}}_\text{+2}
     \underbrace{,\overbracket{{m-3\over 2},{m-3\over 2},\cdots,2,2,1,1}^{|\rho_{m+1,0,m}|,\ldots,|\rho_{2m-3,0,m}|},0,0}_\text{+1}]
\end{split}
\end{equation}
 which after addition and  substituting $r=m+1$ gives the required partition  \eqref{niseven} for $\eh_{0,r}$.
 \end{proof}

\begin{cor}\label{wis}
$\ew_{0,k}$ is the sum  of the multisets  $[2r-3,2r-5,\ldots, 3,1]$ and $[r-1]$.
\end{cor}
\begin{proof}
 Assume  $r$ is even. Then the dual partition  (using the formula in \eqref{dualprn}) is obtained from looking at the indices  above the  numbers in \eqref{niseven}. For example we have $2r-3$ numbers in $\eh_{0,r}$ greater that or equal 1, $2r-5$ numbers greater than or equal $2$ and so on, until we reach  $3$ number greater than or equal $r-1$ and 1 number greater than or equal 1. Thus we find that
\begin{equation}\label{wiseven}
\ew_{0,r}=[2r-3,2r-5,\cdots,r+1,r-1;r-1,r-3,\cdots,3,1]
\end{equation}
Similarly,  in case $r$ is odd we get
\begin{equation}\label{wisodd}
\ew_{0,r}=[2r-3,2r-5,\cdots,r;r-1;r-2,r-4,\cdots,3,1]
\end{equation}
Observe that the formulas above are  independent of $r$ being odd or even since  they can be formed as the sum of the multisets $[2r-3,2r-5,\ldots, 3,1]$ and $[r-1]$.
\end{proof}
 \section{Nilpotent elements of type $D_{2k+2}(a_k)$}

In this section we assume the nilpotent element $e$ is of type $D_{2k+2}(a_k)$ where $k>0$. We will  calculate $\ew_{k,2k+2}$. From the definition, the set $\rho_{i,k,2k+2}$ of roots of height $i$  is given by
$
 \rho_{i,k,2k+2}=\{\sum m_j\alpha_j~:1+\sum_{j\in 2\mathbb{Z}} m_j=i\}.
$

\begin{prop}\label{regnilp}
The set $ \rho_{2k+1,k,2k+2}$ consists of  exactly the highest root and other root divisible by $\alpha_{2k+2}$. Each other set $ \rho_{i,k,2k+2};\, 1<i<2k+1$ has exactly two roots divisible by $\alpha_{2k+2}$. The partition $\eh_{k,2k+2}$ is given by
\begin{eqnarray}
\eh_{k,2k+2} &=&[4k-j-2\lfloor {j\over 2}\rfloor+3: 1\leq j\leq 2k+1]\\\nonumber
&=&[4k+2,4k-1,4k-2,4k-5,...,7,6,3,2]
\end{eqnarray}
\end{prop}
\begin{proof}
The proof is by induction on $k$. The result is true when  $k=1$ since
  \begin{eqnarray*}
  \rho_{1,1,4}&=&\{\alpha_1,\alpha_2,\alpha_4,\alpha_1+\alpha_3,\alpha_2+\alpha_3,\alpha_3+\alpha_4\}\\
      \rho_{2,1,4}&=&\{\alpha_1+\alpha_2+\alpha_3,\alpha_2+\alpha_3+\alpha_4,\alpha_1+\alpha_3+\alpha_4\}\\
     \rho_{3,1,4}&=&\{\alpha_1+\alpha_2+\alpha_3+\alpha_4,\alpha_1+\alpha_2+2\alpha_3+\alpha_4\}\\
      \end{eqnarray*}
and $\eh_{1,4}=[6,3,2]$. We  assume the result is true for $k>0$. For $k+1$, we  have to analyse  the consequence of adding the two roots $\alpha_{2k+3}$ and $\alpha_{2k+4}$ to the set $\Phi^+_{2k+2}$  where $\alpha_{2k+3}$ has weight zero and  $\alpha_{2k+4}$ has weight 2.  Note that
\begin{equation}
 \rho_{1,k+1,2k+4}=\rho_{1,k,2k+2}\cup \{\alpha_{2k+4},\alpha_{2k+3}+\alpha_{2k+4},\alpha_{2k+2}+\alpha_{2k+3}, \alpha_{2k+1}+\alpha_{2k+2}+\alpha_{2k+3}\}
\end{equation}
and so $|\rho_{1,k+1,2k+4}|=| \rho_{1,2k+2}|+4$. Also, for $1<i\leq 2k+1$, if $\gamma_1,\gamma_2$ are the two roots in $\rho_{i-1,k,2k+2}$ divisible by $\alpha_{2k+2}$ and $\sigma_1,\sigma_2$ be the two roots in $\rho_{i,k,2k+2}$ divisible by $\alpha_{2k+2}$. Then   we have

\begin{equation}
\rho_{i,k+1,2k+4}= \rho_{i,k,2k+2} \cup \{\gamma_1+\alpha_{2k+3}+\alpha_{2k+4},\gamma_2+\alpha_{2k+3}+\alpha_{2k+4}, \sigma_1+\alpha_{2k+3},\sigma_2+\alpha_{2k+3}\}
\end{equation}
so $|\rho_{i,k+1,2k+4}|=| \rho_{i,k,2k+2}|+4$ and $\rho_{i,k+1,2k+4}$ contains two roots divisible by $\alpha_{2k+4}$. Now the two roots in $\rho_{2k+1,k,2k+2}$  are
\begin{eqnarray*}
\omega_1&=& \alpha_1+\alpha_2+2\alpha_3+...+2\alpha_{2k}+ \alpha_{2k+1}+\alpha_{2k+2}\\
\omega_2 &=&\alpha_1+\alpha_2+2\alpha_3+...+2 \alpha_{2k+1}+\alpha_{2k+2}
\end{eqnarray*}
This yields
\begin{equation}
 \rho_{2k+2,k+1,2k+4}=\{\omega_1+\alpha_{2k+3}+\alpha_{2k+4},\omega_2+\alpha_{2k+3}+\alpha_{2k+4},\omega_2+\alpha_{2k+2}+\alpha_{2k+3} \}
\end{equation}
and
\begin{equation}
\rho_{2k+3,k+1,2k+4}=\{\omega_2+\alpha_{2k+2}+\alpha_{2k+3}+\alpha_{2k+4}, \alpha_1+\alpha_2+2\alpha_3+... +2\alpha_{2k+3}+\alpha_{2k+4}\}
\end{equation}
hence, $|\rho_{2k+2,k+1,2k+4}|=3$ and $|\rho_{2k+3,k+1,2k+4}|=2$ and the first part of the proposition is true for $k+1$. We conclude that  the partition $\eh_{k+1,2k+4}$ is obtained from  $\eh_{k,2k+2}$ by adding the numbers 4,3 and 2 subject to

 \begin{equation} [\underbrace{4k+2,4k-1,4k-2,4k-5,\cdots, 6,3,2}_\text{+4},\underbrace{0}_\text{+3},\underbrace{0}_\text{+2}]\end{equation}
 which after addition gives  the proposed    height partition  $\eh_{k+1,2k+4}$, i.e.
 \begin{equation}\label{dk}[4(k+1)+2,4(k+1)-1,4(k+1)-2,4(k+1)-5,...,7,6,3,2]\end{equation}
 Finally, numbers appear in $\eh_{k,2k+2}$ are  the union of  two arithmetic series $\{ 4k-2j+4: j~ \mathrm{is~  odd}\}$ and $
            \{4k-2j+3: j~\mathrm{is ~even}\}$. If $j=2m+1$ is odd, we have
            $4k-2j+4=4k-j-2m+3=4k-j-2\lfloor{ j\over 2}\rfloor+3$. This is  the same  formula we get when  $j=2m$ is even.
\end{proof}

\begin{cor}\label{regnilp1}
$\ew_{k,2 k+2}$ is the sum  of the multisets $[2k+1,2k-1,\ldots,3,1]$, $[2k+1,2k-1,\ldots,3,1]$ and $[2k,2k-1,\ldots,1]$.
\end{cor}
\begin{proof}
The transpose  partition  $\ew_{k,2 k+2}$ of $\eh_{k,2k+2}$ is found by analysing the patterns  of  the numbers given in \eqref{dk} and the indices in the definition of $\rho_{j,k,2k+2}$. Thus we have $2k+1$ numbers in $\eh_{k,2k+2}$ greater than or equal 1 and  greater than or equal 2, there is $2k$ numbers greater than or equal 3, and so on until we reach 1 number in $\eh_{k,2k+2}$ greater than or equal to $4k$, $4k+1$ and $4k+2$. This leads to the required partition
\begin{equation}
\begin{split}
 \ew_{k,2k+2}= & [2k+1,2k+1,2k,2k-1,2k-1,2k-1,  \\
    & 2k-2,2k-3,2k-3,2k-3,\cdots,2,1,1,1]
\end{split}
\end{equation}
\end{proof}

\section{Height partition of $D_{r}(a_k)$}
In this section we fix $0<k<\lfloor{r\over 2}\rfloor$ and we  assume the  nilpotent elements $e$ is of type $D_{r}(a_k)$ where $r>2k+2$. We will find the height partition $\eh_{k,r}$. The definition of    $\rho_{j,k,r}$ is given  in  \eqref{hights:def}.

\begin{prop} The set $\rho_{2r-2k-3,k,r}$  contains only the highest root. For $1\leq j<2r-2k-3$, each set $\rho_{j,k,r}$   has exactly one root divisible by $\alpha_r$ except  when $j=r-2k-1,...,r-1$;~ $\rho_{j,k,r}$   has two roots divisible by $\alpha_r$. The partition $\eh_{k,r}$ is given by the  formula in table \ref{htprn1} when $r\leq 4k+3$ and in table \ref{htprn2} when $r>4k+3$
\end{prop}

\begin{table}
\centering
  \begin{tabular}{|c|c|c|}

  \hline &  {Range} & $|\rho{j,k,r}|$\\
    \hline
   1  &{$j\leq r-2k-2$}&$r+2k-2\lfloor {j\over 2}\rfloor$ \\   \hline
    2  &{$r-2k-2<j \leq 2k+1$}& $2r-j-2\lfloor {j\over 2}\rfloor+1$ \\ \hline
     3 &{$j= 2k+2$} &$2r-4k-4$ \\
    \hline
  6 &{$2k+2<j\leq r-1$}&$2r-k-j-\lfloor {j\over 2}\rfloor-1$ \\   \hline

   7 &{$r-1<j\leq 2r-2k-3$}&$ r-k-\lfloor {j\over 2}\rfloor-1$ \\
    \hline
  \end{tabular}
  \caption{The height partition of $D_{r}(a_k)$, $r\leq 4k+3$ }
  \label{htprn1}
\end{table}

\begin{table}
\centering
  \begin{tabular}{|c|c|c|}

  \hline & {Range} & $|\rho{j,k,r}|$\\
    \hline
   1 &{ $j\leq 2k+1$} &$r+2k-2\lfloor {j\over 2}\rfloor$ \\
 \hline

   4& {$j= 2k+2$}
   &$r-1$ \\ \hline
     5 & {$2k+2<j\leq r-2k-2$}&$r+k-\lfloor {j\over 2}\rfloor$ \\ \hline

  6&  {$r-2k-2<j\leq r-1$}&$2r-k-j-\lfloor {j\over 2}\rfloor-1$ \\ \hline

   7& {$r-1<j\leq 2r-2k-3$}&$ r-k-\lfloor {j\over 2}\rfloor-1$ \\

    \hline
  \end{tabular}
  \caption{The height partition of $D_{r}(a_k)$, $r>4k+3$}
  \label{htprn2}
\end{table}

\begin{table}
\centering
\footnotesize
  \begin{tabular}{|c|c|c|c|c|c|}

  \hline &  {Range} & $|\rho{j,k,r}|$& +& $|\rho{j,k,m}|$ &{Range} \\
    \hline
   1  &{$j\leq r-2k-2$}&$r+2k-2\lfloor {j\over 2}\rfloor$ & 1 & $m+2k-2\lfloor {j\over 2}\rfloor$& $j\leq m-2k-2 $ \\ \hline
    2'  &{$j=r-2k-1 $}& $2r-j-2\lfloor {j\over 2}\rfloor+1$ & 1 & $m+2k-2\lfloor {j\over 2}\rfloor$ &$j\leq m-2k-2 $\\   \hline
    2  &{$r-2k-1<j \leq 2k+1$}& $2r-j-2\lfloor {j\over 2}\rfloor+1$ & 2 & $2m-j-2\lfloor {j\over 2}\rfloor+1$ & {$m-2k-2<j \leq 2k+1$}\\ \hline
     3 &{$j= 2k+2$} &$2r-4k-4$ & 2& $2m-4k-4$ & $j=2k+2$\\
    \hline
  6 &{$2k+2<j\leq r-1$}&$2r-k-j-\lfloor {j\over 2}\rfloor-1$ & 2 & $2m-k-j-\lfloor {j\over 2}\rfloor-1$ & {$2k+2<j\leq m-1$} \\   \hline

   7' &{$j=r$}&$ r-k-\lfloor {j\over 2}\rfloor-1$ & 2 & $2m-k-j-\lfloor {j\over 2}\rfloor-1$ & {$2k+2<j\leq m-1$}\\   \hline
   7 &{$r<j\leq 2r-2k-3$}&$ r-k-\lfloor {j\over 2}\rfloor-1$& 1 & $ m-k-\lfloor {j\over 2}\rfloor-1$& {$m-1<j\leq 2m-2k-5$}\\   \hline
   8 &{$2r-2k-3<j\leq 2r-2k-1$}&$ 0$& 1 & $ m-k-\lfloor {j\over 2}\rfloor-1$& {$2m-2k-5<j\leq 2m-2k-3$}\\
    \hline
  \end{tabular}
  \caption{Induction step for  $D_{r}(a_k)$, $r< 4k+3$ }
  \label{ind1}
\end{table}

\begin{table}
\centering
\footnotesize
  \begin{tabular}{|c|c|c|c|c|c|}

  \hline & {Range} & $|\rho{j,k,r}|$ & + & $|\rho{j,k,m}|$& {Range}\\
    \hline
   1 &{ $j\leq 2k+1$} &$r+2k-2\lfloor {j\over 2}\rfloor$ & 1 & $m+2k-2\lfloor {j\over 2}\rfloor$ & { $j\leq 2k+1$} \\
 \hline
   4& {$j= 2k+2$}
   &$r-1$ & 1 & $m-1$  & j=2k+2\\ \hline
     5 & {$2k+2<j\leq r-2k-2$}&$r+k-\lfloor {j\over 2}\rfloor$ & 1 & $m+k-\lfloor {j\over 2}\rfloor$  & {$2k+2<j\leq m-2k-3$}\\ \hline
  6'&  {$j=r-2k-1$}&$2r-k-j-\lfloor {j\over 2}\rfloor-1$ & 1 &  $m+k-\lfloor {j\over 2}\rfloor$ & $j= m-2k-2$\\ \hline
  6&  {$r-2k-1<j\leq r-1$}&$2r-k-j-\lfloor {j\over 2}\rfloor-1$ & 2 &$2m-k-j-\lfloor {j\over 2}\rfloor-1$ &{$m-2k-2<j\leq m-2$}\\ \hline
   7'& {$j=r$}&$ r-k-\lfloor {j\over 2}\rfloor-1$ & 2 & $2m-k-j-\lfloor {j\over 2}\rfloor-1$ & $m-1$\\ \hline
   7& {$r<j\leq 2r-2k-3$}&$ r-k-\lfloor {j\over 2}\rfloor-1$ & 1 & $ m-k-\lfloor {j\over 2}\rfloor-1$ & {$m-1<j\leq 2m-2k-5$}\\ \hline
   8& {$2r-2k-3<j\leq 2r-2k-1 $}&$0$ & 1 & $ m-k-\lfloor {j\over 2}\rfloor-1$ & {$2m-2k-5<j\leq 2m-2k-3$} \\

    \hline
  \end{tabular}
  \caption{Induction step for $D_{r}(a_k)$, $r>4k+3$}
  \label{ind2}
\end{table}
\begin{proof}
The proof is by induction on $r-2k-2$. When  $r-2k-2=1$, we obtain   $|\rho_{j,k,r}|$ using lemma \ref{basic} and proposition \ref{regnilp} as $r-1=2k+2$. Thus
  \begin{eqnarray}
  |\rho_{1,k,r}|&=&|\rho_{1,k,2k+2}|+1=4k+3=r+2k\\\nonumber
  |\rho_{j,k,r}|&=&|\rho_{j,k,2k+2}|+2 = 4k-j-2\lfloor {j\over 2}\rfloor+5;~1<j<2k+1\\\nonumber
  |\rho_{2k+2,k,r}|&=&2; ~~
  |\rho_{2k+3,k,r}|=1
  \end{eqnarray}
  which are  the formula in  table \ref{htprn1} for  $r=2k+3$. Also, $\rho_{2k+3,k,r}$ contains only the highest root and  $\rho_{j,k,r}$ has exactly 2 roots divisible by $\alpha_{r}$ when  $1<j<2k+3$. Hence, the statement  is true for  $r-2k-2=1$. Assume the formula is true when $1\leq r-2k-2<2k+1$.  For  $r+1$, we  summarize in table \ref{ind1} the inductive step. In the third column of this table, we give the number of roots to be  added to $\rho_{j,k,r}$ in order to obtain $\rho_{j,k,r+1}$. The first column list the range of $j$ with attention to the value of $j$ where $\rho_{j,k,r}$ contains two roots divisible by $\alpha_{r}$. The second and fourth columns are the formula of the size of the sets before and after adding the roots. In the fifth column we altred the range of $j$ using $m=r+1$ to show that the formlas are still valid for $r+1$. Note that the addition by 2 means that there will be two roots in $\rho_{j,k,m}$ which are divisible by $\alpha_{r+1}$. This proves the statement for  $r\leq 4k+3$. Note that for $r=4k+4$, using again lemma \ref{basic} and the formula of table  \ref{htprn1} we get
    \begin{eqnarray}
  |\rho_{j,k,r}|&=&|\rho_{j,k,r-1}|+1=r+2k-2\lfloor {j\over 2}\rfloor;~~ j\leq 2k+1\\\nonumber
  |\rho_{2k+2,k,r}|&=&|\rho_{2k+2,k,r-1}|+1 =2(r-1)-4k-4+1= r-1\\\nonumber
  |\rho_{j,k,r}|&=&|\rho_{j,k,r-1}|+2 =2r-k-j-\lfloor {j\over 2}\rfloor-1; ~~{2k+2<j\leq r-2}\\\nonumber
  |\rho_{r-1,k,r}|&=&|\rho_{r-1,k,r-1}|+2 = 2r-k-j-\lfloor {j\over 2}\rfloor-1;~~ j=r-1\\\nonumber
  |\rho_{j,k,r}|&=&|\rho_{j,k,r-1}|+1 =r-k-\lfloor {j\over 2}\rfloor-1; ~~{r-1<j\leq 2r-2k-3}
  \end{eqnarray}
  which is exactly the formula in table \ref{htprn2} for $r=4k+4$ and the statement is true in this case. Then we assume for induction purposes that  the statement is true for some $r>4k+3$. Then the inductive step for  $r+1$ is proved by constructing table \ref{ind2} which is similar to table \ref{ind1}. This proves the theorem.
   \end{proof}


\begin{table}\footnotesize
  \centering
\begin{tabular}{|c|c|c|}
  \hline
  $\rho_{i,k,r}$ & $r$ is odd & $r$ is even \\   \hline
 $\rho_{1,k,r}$ & $r+2k$ &  $r+2k$  \\\hline
  $\rho_{2,k,r}$ & $r+2k-2$ &  $r+2k-2$  \\   \hline
  $\rho_{3,k,r}$ & $r+2k-2$ &  $r+2k-2$  \\   \hline
  $\rho_{4,k,r}$ & $r+2k-4$ &  $r+2k-4$  \\   \hline
  $\rho_{5,k,r}$ & $r+2k-4$ &  $r+2k-4$  \\   \hline
  $\vdots$ & $\vdots$ &  $\vdots$  \\   \hline
  $\rho_{2k,k,r}$ & $r$ &  $r$  \\   \hline
  $\rho_{2k+1,k,r}$ & $r$ &  $r$  \\   \hline
  $\rho_{2k+2,k,r}$ & $r-1$ &  $r-1$  \\   \hline
  $\rho_{2k+3,k,r}$ & $r-1$ &  $r-1$  \\   \hline
  $\vdots$ & $\vdots$ &  $\vdots$  \\   \hline
  $\rho_{r-(2k+3),k,r}$ & ${r+1\over 2}+2k+1$ &  $\vdots$  \\   \hline
  $\rho_{r-(2k+2),k,r}$ & ${r+1\over 2}+2k+1$ &  ${r\over 2}+2k+1$  \\   \hline
   $\rho_{r-(2k+1),r}$ & ${r+1\over 2}+2k$ &  ${r\over 2}+2k+1$  \\   \hline
    $\rho_{r-2k,k,r}$ & ${r+1\over 2}+2k-1$ &  ${r\over 2}+2k-1$  \\   \hline
     $\rho_{r-2k+1,k,r}$ & ${r+1\over 2}+2k-3$ &  ${r\over 2}+2k-2$  \\   \hline
      $\rho_{r-2k+2,k,r}$ & ${r+1\over 2}+2k-4$ &  ${r\over 2}+2k-4$  \\   \hline
       $\rho_{r-2k+3,k,r}$ & ${r+1\over 2}+2k-6$ &  ${r\over 2}+2k-5$  \\   \hline
        $\vdots$ & $\vdots$ &  $\vdots$  \\   \hline
         $\rho_{r-2,k,r}$ & ${r+1\over 2}-k-2$ &  ${r\over 2}-k$  \\   \hline
         $\rho_{r-1,k,r}$ & ${r+1\over 2}-k$ &  ${r\over 2}-k+1$  \\   \hline
          $\rho_{r,k,r}$ & ${r+1\over 2}-k-1$ &  ${r\over 2}-k-1$  \\   \hline
           $\rho_{r+1,r}$ & ${r+1\over 2}-k-2$ &  ${r\over 2}-k-1$  \\   \hline
            $\rho_{r+2,k,r}$ & ${r+1\over 2}-k-2$ &  ${r\over 2}-k-2$  \\   \hline
             $\rho_{r+3,k,r}$ & ${r+1\over 2}-k-3$ &  ${r\over 2}-k-2$  \\   \hline
              $\rho_{r+4,k,r}$ & ${r+1\over 2}-k-3$ &  ${r\over 2}-k-3$  \\   \hline
              $\vdots$ & $\vdots$ &  $\vdots$  \\   \hline
             $\rho_{2r-2k-6,k,r}$ & $2$ &  $2$  \\   \hline
        $\rho_{2r-2k-5,k,r}$ & $2$ &  $2$  \\   \hline
        $\rho_{2r-2k-4,k,r}$ & $1$ &  $1$  \\   \hline
        $\rho_{2r-2k-3,k,r}$ & $1$ &  $1$  \\   \hline
\end{tabular}
\caption{The height partition when $r>4k+3$}
\label{heit1}
\end{table}

 \begin{table}
 \footnotesize
  \centering
    \begin{tabular}{|c|c|}
   \hline
  $\eh_{k,r}$ & $\ew_{k,r}$ \\ \hline
   $|\rho_{2r-2k-3,k,r}|,...,|\rho_{r,k,r}|$& $2r-2k-3,2r-2k-5,...,r+2,r$\\ \hline
   \multirow{ 2}{*} {$|\rho_{r-1,k,r}|,...,|\rho_{r-(2k+1),k,r}|$} & $r-1,r-2,r-2,r-3,r-4,r-4,$\\ &$...,r-2k+1, r-2k,r-2k, r-2k-1 $\\
      \hline
     $|\rho_{r-(2k+2),k,r}|,...,|\rho_{2k,k,r}|$& $r-2k-2, r-2k-4,...,2k+3,2k+1$\\\hline
     $|\rho_{2k,k,r}|,...,|\rho_{1,k,r}|$&$2k-1,2k-1, 2k-3,2k-3,...,3,3,1,1$\\\hline
  \end{tabular}
  \caption{The weight partition, $r>4k+3$ is odd}
  \label{wtodd1}
  \end{table}

  \begin{table}
 \footnotesize
  \centering
 \begin{tabular}{|c|c|}
   \hline
 $\eh_{k,r}$ &   $\ew_{k,r}$\\ \hline
   $|\rho_{2r-2k-3,k,r}|,...,|\rho_{r,k,r}|$& $2r-2k-3,2r-2k-5,...,r+3,r+1$\\ \hline
   \multirow{ 2}{*} {$|\rho_{r-1,k,r}|,...,|\rho_{r-(2k+2),k,r}|$ }& $r-1,r-1,r-2,r-3,r-3,r-4,$\\
   & $...,r-2k, r-2k-1,r-2k-1 $\\
      \hline
     $|\rho_{r-(2k+3),k,r}|,...,|\rho_{2k,k,r}|$& $r-2k-3, r-2k-5,...,2k+3,2k+1$\\\hline
     $|\rho_{2k-1,k,r}|,...,|\rho_{1,k,r}|$&$2k-1,2k-1, 2k-3,2k-3,...,3,3,1,1$\\\hline
  \end{tabular}
  \caption{The weight partition, $r>4k+3$ is even}
  \label{wteven1}
\end{table}

\section{The weight partition of $D_{r}(a_k)$}
Let $e$ be a nilpotent element of type $D_{r}(a_k)$ and we assume that $r>4k+3$ and $k>0$. The case of $r< 4k+3$ can be treated similarly. Using  the formulas in table \ref{htprn2}, the height partition $\eh_{k,r}$ is given  in table \ref{heit1} where we need to distinguish between $r$ is even and $r$ is odd to get rid of the floor function. Also, we have to indicate carefully the  indices of $j$ where a pattern in the list of numbers break.

Assume that $r$ is odd. Then as illustrated in previous sections the transpose partition is found by looking on the indices $j$ in $\rho_{j,k,r}$. Thus, from  column 2 of table \ref{heit1}, we have $2r-2k-3$ numbers in $\eh_{k,r}$ greater than or equal 1, $2r-2k-5$ greater than or equal $2$, and so on until we reach 1 number is greater than or equal to $r+2k-1$, $r+2k$. Similarly for the case $r$ is even.  This leads to table \ref{wtodd1} and \ref{wteven1} where we give  the weight partition $\ew_{k,r}$ in the second column and in column 1  sequences  of numbers of $\eh_{k,r}$ which have some patterns.

Now we are able to prove the main theorem about all semiregular nilpotent elements in Lie algebra of type $D_r$.

\begin{thm}\label{main2}
The weight partition  $\ew_{k,r}$ of a nilpotent element of type $D_r(a_k)$ is the sum of the multisets $ [1,3,5,\cdots,2r-2k-3]$, $[1,3,5,\cdots,2k-1] $ and  $[r-1,r-2,r-3,\cdots,r-2k-1]$.
\end{thm}

\begin{proof}
The case of regular nilpotent element ($k=0$),  $\ew_{k,r}$ is obtained in corollary  \ref{wis}.  For $k>0$ and $r=2k+2$, $\ew_{k,r}$ is given in corollary  \ref{regnilp1}.  Assume  $k>0$ and $r>4k+3$. Then $\ew_{k,r}$ is given in \ref{wtodd1} and \ref{wteven1}. However, we observe that the resulting partition $\ew_{k,r}$ is independent of $r$ being odd or even. In both tables (\ref{wtodd1} and \ref{wteven1}), $\ew_{k,r}$ is the sum of  $[2r-2k-3,\ldots,5,3,1]$, $[2k-1,\ldots,5,3,1]$, which consist of series of odd numbers, and the series of numbers $[r-1,r-2,\ldots,r-2k-1]$. In case $k>0$ and $r\leq 4k+3$, we get the result by obtaining similar tables as   \ref{wtodd1} and \ref{wteven1}. We omit the proof in this case and  we leave it  for the reader to verify.
\end{proof}    

We observe that the formula of $\ew_{k,r}$ for a general nilpotent element of type $D_r(a_k)$ is  easy to memorize by recalling that in the case of special linear Lie algebra $so_{2r}$ this nilpotent element $D_r(a_k)$ corresponds to the partition $[2r-2k-1,2k+1]$ of $2r$.

\noindent Yassir Dinar

\noindent Sultan Qaboos University, Muscat, Oman

\noindent dinar@squ.edu.om

\begin{thebibliography}{00}
\bibitem{arbour} Arbour, B., Doković, D.,
On semiregular nilpotent orbits in simple Lie algebras of type D, J. Algebra Appl. 1, no. 3, 341–356 (2002).

\bibitem{bour} Bourbaki, N.,  Lie groups and Lie algebras, Chapters 4–6, Springer-Verlag, Berlin, ISBN: 3540426507 (2002).
\bibitem{car} Roger W. Carter,{\it Simple groups of Lie type. Wiley Classics Library}. A Wiley-Interscience Publication, ISBN: 04715068342002 (1989).

 \bibitem{COLMC} David H. Collingwood and  William M. McGovern, {\it Nilpotent orbits in semisimple Lie algebras}. Van Nostrand Reinhold Mathematics Series. ISBN: 0-534-18834-6 (1993).

  \bibitem{mypaper1} Dinar, Yassir, Frobenius manifolds from regular classical W-algebras. Advances in Mathematics, Volume 226,
Issue 6, Pages 5018-5040 (2011).
\bibitem{mypaper5} Dinar, Yassir, Classical W-algebras and Frobenius manifolds related to Liouville completely integrable systems, arXiv:1911.00271 (2019).
\bibitem{feher}Fehér, L., O'Raifeartaigh, L., Ruelle, P., Tsutsui, I.,
On the completeness of the set of classical $W$-algebras obtained from DS reductions-algebras obtained from DS reductions
Comm. in Math. Phy., 162 (2), pp. 399-431.
(1994).
\bibitem{JHum} Humphreys, James E, Reflection groups and Coxeter groups,
Cambridge Studies in Advanced Mathematics, 29, ISBN: 0-521-37510-X   (1990).
\bibitem{kostBetti} Kostant, B., The principal    three-dimensional subgroup and the Betti numbers of a complex simple Lie group, Amer. J. Math. 81, 973(1959).
\bibitem{kostpoly} Kostant, B., Lie group representations on polynomial rings. Amer. J. Math. 85,  327–404 (1963).
\bibitem{Lehn} Lehn, M., Namikawa, Y., Sorger, Ch., Slodowy slices and universal Poisson deformations. Compos. Math. 148, no. 1, 121–144 (2012).
\bibitem{richard} Richardson, R. W.,
Derivatives of invariant polynomials on a semisimple Lie algebra, Miniconference on harmonic analysis and operator algebras (Canberra, 1987), 228–241,
Proc. Centre Math. Anal. Austral. Nat. Univ., 15, Austral. Nat. Univ., Canberra, (1987).
 \bibitem{sldwy1}   Slodowy P., {\it Simple singularities and simple algebraic groups}, Lect.Notes in Math. 815 Springer Verlag, Berlin.(1980).
\end{thebibliography}
\end{document}